%% file: main.tex
\documentclass{amsart}

\input{packages}
\input{commands}

\title{Geodesic and conformally Reeb vector fields on flat 3-manifolds}
\author{Tilman Becker}
\address{Mathematisches Institut, Universit\"at zu K\"oln, Weyertal 86-90, 50931 K\"oln, Germany}
\email{tibecker@math.uni-koeln.de}

\begin{document}

\begin{abstract}
A unit vector field on a Riemannian manifold $M$ is called geodesic if all of its integral curves are geodesics. We show, in the case of $M$ being a complete flat 3-manifold not equal to $\E^3$, that every such vector field is tangent to a 2-dimensional totally geodesic foliation.  Furthermore, it is shown that a geodesic vector field $X$ on a closed orientable complete flat 3-manifold is (up to rescaling) the Reeb vector field of a contact form if and only if there is a contact structure transverse to $X$ that is given as the orthogonal complement of some other geodesic vector field. An explicit description of the lifted contact structures (up to diffeomorphism) on the 3-torus is given in terms of the volume of $X$. Finally, similar results for non-closed flat 3-manifolds are discussed.
\end{abstract}
\maketitle
\section{Introduction}
Let $M$ be a manifold and $X$ a nowhere vanishing vector field on $M$. Then $X$ is called \emph{geodesible} if there exists a Riemannian metric $g$ on $M$ such that $X$ is of unit length and all of its integral curves are geodesics with respect to $g$. If the metric $g$ is fixed, $X$ is called \emph{geodesic}. Analogously, a one-dimensional foliation is called geodesible (resp.\ geodesic) if there is a geodesible (resp.\ geodesic) vector field spanning it. Note that, by definition, a geodesible foliation is always orientable. 

The class of geodesible foliations or vector fields is a large and interesting one. Examples include: 
\begin{itemize}
\item[(1)] Killing vector fields of unit length,
 \item[(2)] vector fields that admit a global closed (hyper-)surface of section,
 \item[(3)] Reeb vector fields of contact forms or stable Hamiltonian structures.
\end{itemize}
For a discussion of these examples and geodesible foliations in general, see \cite{gluck:1980}. Of particular interest are geodesic foliations of Riemannian manifolds of constant sectional curvature. The study of such foliations was initiated by Gluck and Warner \cite{gluckwarner:1983} in the 1980s, who described the possible ways the 3-sphere can be fibered by great circles. More than two decades later, Salvai \cite{salvai:2009} and Harrison \cite{harrison:2016,harrison:2021} have given similar characterizations of line fibrations of $\E^3$ (that is, $\R^3$ with the Euclidean metric). More generally, one may consider geodesic foliations of arbitrary flat $3$-manifolds. The following is the first result of the present paper.

\begin{thm}\label{thm:main1}
Let $M$ be a complete flat 3-manifold not equal to $\E^3$. Then any one-dimensional oriented geodesic foliation of $M$ is tangent to a two-dimensional totally geodesic foliation.
\end{thm}
\begin{rmk}
(i) Note that we do not assume $M$ to be oriented or closed. Furthermore, the statement is false for geodesic foliations of $\E^3$. In fact, there exist fibrations of $\E^3$ by pairwise non-parallel oriented lines (so-called \emph{skew fibrations}), see \cite{harrison:2016}.

(ii) As pointed out by the referee, Theorem \ref{thm:main1} may also be proved using Theorem 1 (d) and a slight variation of Lemma 17 in \cite{harrison:2021}.

\end{rmk}

We will prove Theorem \ref{thm:main1} in Section \ref{section:geod} using the fact that every complete flat $3$-manifold that is not equal to $\E^3$ can be written as the quotient $\E^3 / \Gamma$, where $\Gamma$ is a subgroup of the isometry group of $\E^3$ that contains either a translation or a screw motion. We then look at the lifted geodesic foliation of $\E^3$ (i.e., a fibration by oriented lines), which must be invariant under the action of $\Gamma$. Using some elementary geometric arguments, we will see that this forces the line fibration to be tangent to a fibration by affine planes, which is equivalent to the statement in the theorem.  

There is also an interesting relation between geodesic foliations and contact structures. Recall that a contact structure on a $(2n+1)$-dimensional manifold $M$ is a maximally non-integrable hyperplane field $\xi \subset TM$. That is, if we write $\xi$ locally as the kernel of a 1-form $\alpha$, then the \emph{contact condition} $\alpha \wedge (\drm \alpha)^n \neq 0$ must hold everywhere. Any such $\alpha$ is called a \emph{contact form} defining the contact structure $\xi$. Now given a geodesic foliation, one can consider the hyperplane field given by the orthogonal complement of the corresponding fiber at each point. If this hyperplane field defines a contact structure, one says that the contact structure is \emph{induced} by the geodesic foliation. Perhaps the most basic example is the Hopf fibration of the standard 3-sphere, which is a great circle fibration inducing the standard contact structure. Gluck \cite{gluck:2018} has recently shown that any great circle fibration of the 3-sphere induces a contact structure, and that any such contact structure is diffeomorphic to the standard one. However, this does not hold in dimensions $\geq 5$, see \cite{gluckyang:2019}. The situation in the Euclidean case is a little more restrictive, as was shown by Harrison \cite{harrison:2019,harrison:2021}: A geodesic vector field $X$ spanning a line fibration of $\E^3$ induces a contact structure if and only if $\rank \nabla X \geq 1$, where $\nabla$ is the Levi-Civita connection. Similar to the case of $S^3$, it has also been shown that any of these contact structures is diffeomorphic to the standard one, see \cite{harrison:2019} and  \cite{beckergeiges:2021}.

To every contact form $\alpha$ there is associated a specific vector field, called the \emph{Reeb vector field} of $\alpha$ (denoted by $R_{\alpha}$). It is the unique vector field spanning the one-dimensional kernel of $\drm \alpha$, normalized so that $\alpha(R_{\alpha}) = 1$.
If a geodesic vector field $X$ induces a contact structure, there is always a corresponding contact form whose Reeb vector field is given by $X$, namely, $\alpha = i_X g$, where $g$ is the underlying Riemannian metric.

Now say we are given a Riemannian 3-manifold $M$ of constant sectional curvature equal to $1$. Then $M$ is the quotient of $S^3$ by the action of some finite subgroup \hbox{$\Gamma < \isom(S^3)$} of isometries. If $X$ is a geodesic vector field of $M$, we can lift it to a vector field on $S^3$ spanning a great circle fibration. By Gluck's result, the lifted vector field induces a contact structure. Since orthogonal complements are preserved under the action of $\Gamma$, this implies that $X$, too, induces a contact structure. In particular, $X$ is the Reeb vector field of some contact form. Similarly, using Harrison's result, a geodesic vector field $X$ on a flat 3-manifold induces a contact structure if and only if $\rank \nabla X \geq 1$, and $X$ is also Reeb in this case. However, consider for example the constant geodesic vector field $\partial_z$ on $\E^3$. This clearly does not induce a contact structure (the orthogonal complement being a constant plane field), but it is the Reeb vector field of a contact form, namely the standard one given by $\dz + x \, \dy$. That is, unlike in the case of positive constant curvature, the class of geodesic Reeb vector fields on flat 3-manifolds does not coincide with the class of geodesic vector fields that induce contact structures. 
The natural question then is:
\begin{center}
What is a (necessary and sufficient) criterion for a geodesic vector field on a flat $3$-manifold to be conformally Reeb?
\end{center}
Here, a vector field $X$ is said to be \emph{conformally Reeb} if there is a contact form $\alpha$ and a positive function $\lambda$ such that $X = \lambda \, R_{\alpha}$ (see also \cite{prasad:2022}).
This question is also motivated by Example (3) above: Every Reeb vector field is geodesible, but the converse is not true. Indeed, consider for example the manifold $S^2 \times S^1$ and the geodesible vector field $X = \partial_{\varphi}$, where $\varphi$ is the angular coordinate of the $S^1$-factor. This vector field cannot be Reeb (not even up to rescaling): If $\alpha$ were a contact form whose Reeb vector field is parallel to $X$, then $\drm \alpha$ would restrict to an exact area form on $S^2 \times \set{\text{point}}$, which is not possible due to Stokes' theorem. Generally, vector fields that admit closed global surfaces of section are always geodesible (Example (2) above) but never Reeb.
The question can then be seen as a special case of the more general question of whether or not a given geodesible vector field is conformally Reeb.
We give an answer to the above question for geodesic vector fields on closed orientable complete flat $3$-manifolds. Recall that by the classical Bieberbach theorems \cite{bieberbach:1911,bieberbach:1912} (see also \cite{wolf:1984}), any such $3$-manifold $M$ can be written, up to affine diffeomorphism (that is, a diffeomorphism preserving the Levi-Civita connection), as the quotient $M = T^3 / \Gamma$, where $\Gamma < \isom(T^3)$ is a finite subgroup of isometries acting freely on $T^3$, and $T^3$ is the standard flat 3-torus. Then, a geodesic vector field $X$ on $M$ can be lifted to a geodesic vector field $\xlif$ on $T^3$. We obtain the following result.

\begin{thm}\label{thm:main2}
Let $X$ be a geodesic vector field on a closed orientable complete flat $3$-manifold $M$. Then $X$ is conformally Reeb for a contact form $\alpha$ if and only if there is a geodesic vector field $Y$ on $M$ inducing a contact structure $\xi$ such that $X$ is everywhere transverse to $\xi$.

In this case, writing $M$ as $M =T^3/ \Gamma$ with $\Gamma < \isom(T^3)$, there is a fibration $\zeta \col T^3 \to S^1$ whose fibers are totally geodesic $2$-tori such that the lifted vector fields $\xlif$ and $\ylif$ are tangent to the fibers of $\zeta$. Furthermore, the lifted contact structures $\ker \, \alpha_T$ and $\xi_T$ on $T^3$ are both diffeomorphic to
\[
\ker \, \left(\sin\left(\frac{\vol_{X} \, \abs{\Gamma}}{A} \zeta\right) \mc{E}^1 + \cos\left(\frac{\vol_X \abs{\Gamma}}{A} \zeta\right) \mc{E}^2\right),
\]
where
\begin{itemize}
    \item $\mc{E}^1$ and $\mc{E}^2$ are 1-forms dual to a global orthonormal parallel frame $(E_1, E_2)$ spanning the fibers of $\zeta$,
    \item $A := \int_{\zeta^{-1}(a)} \mc{E}^1 \wedge \mc{E}^2$ is the (Euclidean) area of a typical fiber,
    \item $\vol_X$ is the \emph{volume} of $X$.
\end{itemize}
\end{thm}
See section \ref{section:volume} for the definition of the volume of a geodesible vector field. 

\begin{cor}\label{cor:std}
Let $X$ be a geodesic vector field on a closed flat 3-manifold. If $X$ is conformally Reeb for a contact form $\alpha$, then the lifted contact structure $\ker \alphtil$ on $\R^3$ is diffeomorphic to the standard contact structure $\ker \, (\drm z + x\, \drm y).$
\end{cor}

Theorem \ref{thm:main2}, and subsequently Corollary \ref{cor:std}, will be proven in section \ref{section:main} using the characterization in Theorem \ref{thm:main1}.

\begin{rmk}
Both the "if"- and the "only if"-part of the first statement of Theorem \ref{thm:main2} are false in general for non-closed manifolds, see section \ref{section:open}.
\end{rmk}

\subsection*{Acknowledgements} I want to thank my advisor Hansj\"org Geiges, as well as Murat Sa\u{g}lam, for many helpful discussions regarding this work. I would also like to thank the referee for the very detailed report, and for pointing out the alternative proof of Theorem \ref{thm:main1}.

This work is part of a project of the SFB/TRR 191 'Symplectic Structures in Geometry, Algebra and Dynamics', funded by the DFG (Projektnummer 281071066 -- TRR 191).

\section{Proof of Theorem \ref{thm:main1}}\label{section:geod}
Let $M$ be a complete flat 3-manifold not equal to $\E^3$. Then $M$ can be identified with $\E^3 / \Gamma$, where $\Gamma < \isom(\E^3)$ is a non-trivial discrete subgroup of isometries acting freely. Now let $\fol$ be a (one-dimensional, oriented) geodesic foliation of $M$, spanned by a unit vector field $X$. Then $\fol$ lifts to a geodesic foliation $\lfol$ of $\E^3$, spanned by the lifted vector field $\xdlif$. Here, we view $\lfol = \set{\ell}$ just as a set of lines. For a point $p \in \E^3$, denote by $\ell_p \in \lfol$ the fiber through $p$.  Note that $\lfol$ must be invariant under the action of $\Gamma$, that is, $\ell_{\gamma(p)} = \gamma(\ell_p)$ for every $\gamma \in \Gamma$ and $p \in \E^3$. Now it clearly suffices to prove the statement for the lifted foliation $\lfol$, since the covering map $\pi \col \E^3 \to M$ is locally isometric. That is, we have to show that the fibration $\lfol$ of $\E^3$ by oriented lines is tangent to a fibration by affine planes. A line fibration of this type is also called \emph{one-parameter}, cf. \cite{harrison:2021}. To do so, let us take a closer look at the group $\Gamma < \isom(\E^3)$. It is well known that every isometry of $\E^3$ (also called \emph{Euclidean motion}) is given by the composition of a reflection in a plane or rotation about some axis, and some (perhaps trivial) translation. Then one can easily see that any fixed-point free Euclidean motion must be one of the following three: 
\begin{itemize}
    \item a translation;
    \item a screw motion, i.e. rotation about some axis followed by translation in the direction of this axis;
    \item a glide reflection, i.e. reflection in some plane followed by translation parallel to this plane.
\end{itemize}
Note that applying a glide reflection twice yields a (pure) translation again. Hence, we may assume that the group $\Gamma$ contains a non-trivial translation or screw motion. We will treat these two cases separately. 

\textbf{\underline{First case ($\Gamma$ contains a translation)}:} Assume that there is some $T_v \in \Gamma$, where $T_v$ is the translation by some vector $v \in \R^3$. 
If $\xdlif$ is constant, there is nothing to prove. Otherwise, there is a point $p_0 \in \E^3$ such that $\ell_{p_0}$ does not point in the direction of $v$. Let $P \subset \E^3$ be the affine plane through $p_0$ spanned by $v$ and the cross-product $\xdlif(p_0) \times v$. Then $P$ is transverse to $\ell_{p_0}$, so we can consider the projection $\pi \col \E^3 \to P$ onto $P$ in the direction of $\ell_{p_0}$. Define a vector field $Y$ on $P$ by
\[
Y(p) := \drm \pi_p(\xdlif(p)),
\]
and denote by $\ell^Y_p$ the line in $P$ spanned by $Y(p)$. Note that $\ell^Y_p$ is just given by the projected line $\pi(\ell_p)$.
The $\Z$-action on $\E^3$ generated by the translation $T_v$ restricts to a $\Z$-action on $P$, and $Y$ is invariant under this action. Now partition $P$ as $P = A \sqcup B$, where $A = \set{Y \neq 0}$ and $B = \set{Y = 0}$. Note that $B$ is precisely the set of points $p \in P$ for which $\ell_p$ is parallel to $\ell_{p_0}$. Therefore, we may assume that $A \neq \emptyset$, for otherwise, $\xdlif$ is constant and therefore trivially one-parameter. Also, if $Y(p) \neq 0$ at some point $p \in P$, then $\ell^Y_p$ must be disjoint from $B$. Indeed, if there were a point $q \in B \cap \ell^Y_p$, then the fiber $\ell_p$ would intersect $\ell_q$ transversely, which is of course not possible.

Now we consider two cases. First, assume that there is a point $q \in A$ for which $Y(q)$ is parallel to $v$. Then $Y$ must be parallel to $Y(q)$ on the whole line $\ell^Y_{q}$. Indeed, if that were not the case, then the set of lines $\set{\ell^Y_p \col p \in \ell^Y_{q}}$ would fill out a cone that intersects $p_0 + \Z \, v$, see Figure \ref{fig:y}. In particular, there would be some line $\ell^Y_p$ intersecting a point in $B$, which is not possible, as we have seen above. For the same reason, $Y$ must be non-vanishing on $\ell^Y_{q}$ (in fact, we have that $Y(p) = Y(q)$ for all $p \in \ell^Y_q$). It follows that the plane spanned by $\ell^Y_{q}$ and $\xdlif(q)$ is fibered by pairwise parallel lines in $\mc{F}$. The same holds for every parallel translate of that plane, and we conclude that $\lfol$ is one-parameter. 

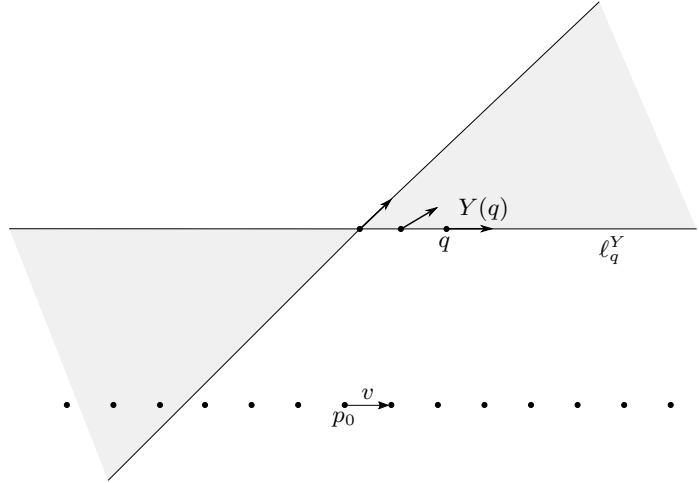
\begin{figure}[ht]

\labellist
\small\hair 2pt
\pinlabel $p_0$ [t] at 340 100
\pinlabel $v$ [b] at 363 103
\pinlabel $q$ [t] at 437 270
\pinlabel $Y(q)$ [b] at 475 275
\pinlabel $\ell_{q}^Y$ [t] at 600 270

\endlabellist
\centering
\includegraphics[scale=0.4]{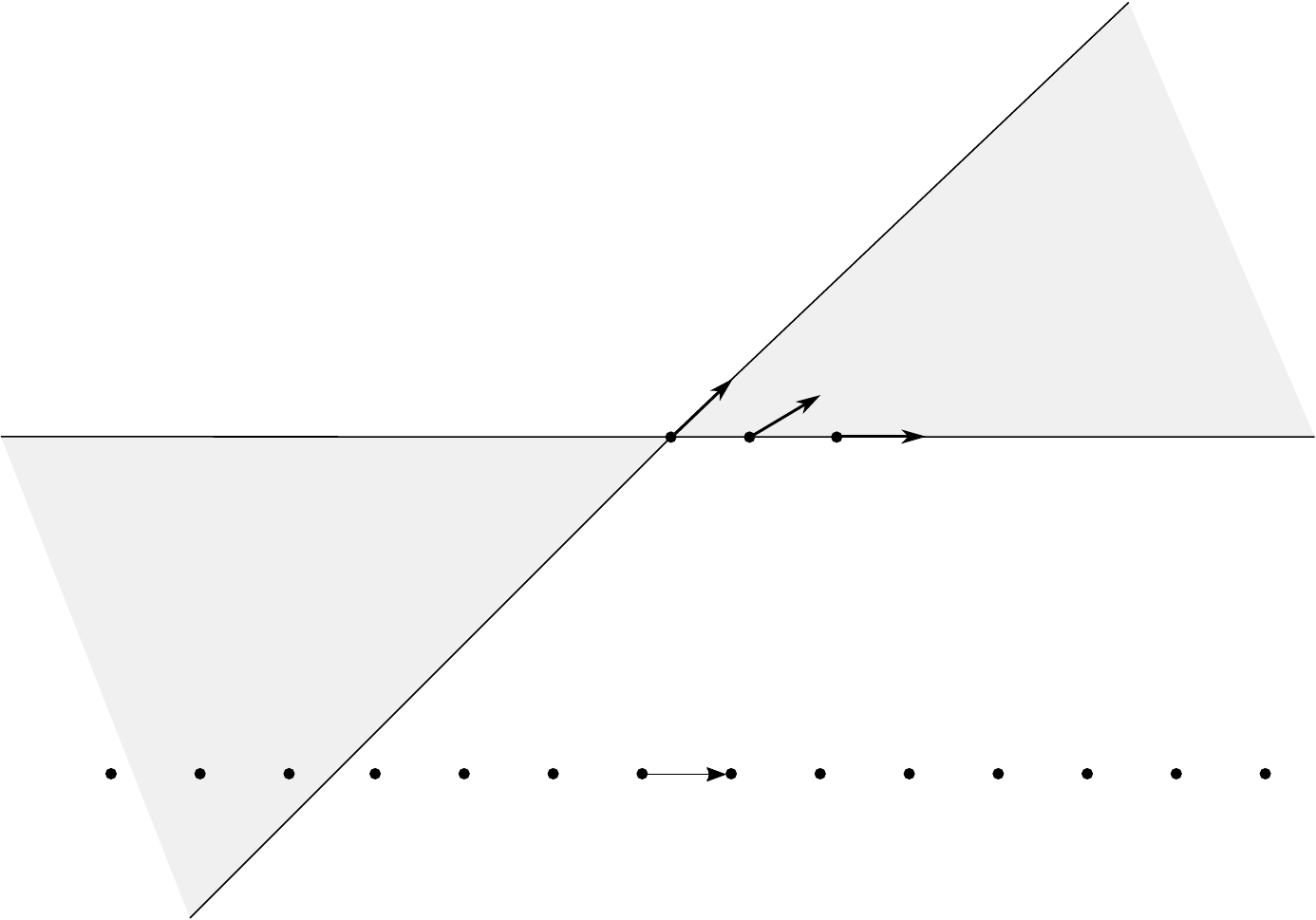}
\caption{The set of lines spanned by $Y$ contains the grey cylinder, which intersects the set of points $\set{p_0 + \Z \, v}$.}
\label{fig:y}
\end{figure}

Thus, we may assume that $Y$ is nowhere parallel to $v$. Let $Q \subset \E^3$ be the affine plane through $p_0$ spanned by $v$ and $\xdlif(p_0)$. Then $Q$ contains infinitely many fibers of $\mc{F}$ parallel to $\ell_{p_0}$, namely, the fibers over points in $p_0 + \Z \, v$. Note that these points correspond to points in $B \subset P$. If $Q^\prime$ is any other affine plane parallel to $Q$, then there must be fibers contained in $Q^\prime$ as well. To see this, denote by $U_Q$ and $U_{Q^\prime}$ the set of points in $Q$ and $Q^\prime$, respectively, where $\xdlif$ is transverse to $Q$ (resp. $Q^\prime$). Then, the flow of $\xdlif$ maps $U_{Q}$ diffeomorphically to $U_{Q^\prime}$. But since there is a $\Z$-family of fibers tangent to $Q$, we see that $U_{Q}$ is either empty or disconnected, so the same must be true for $U_{Q^\prime}$. In particular, $U_{Q^\prime} \neq Q^\prime$, so that there must be fibers in $\mc{F}$ tangent to $Q^\prime$. All of these fibers must be parallel to $\ell_{p_0}$, for otherwise $Y$ is parallel to $v$ (and non-zero) somewhere, and we are in the first case again. Furthermore, the translates of these fibers by integer multiples of $v$ are again fibers of $\mc{F}$ contained in $Q^\prime$. But then every disc of radius $> \abs{v}$ in $P$ must intersect $B$ in at least one point. Now by an argument similar to the one in the first case, we see that for every point $q \in A$ and $p \in \ell^Y_{q}$, we have that $Y(p) = Y(q)$, and we conclude that $\lfol$ is one-parameter. 

\textbf{\underline{Second case ($\Gamma$ contains a screw motion)}:} Assume that $\Gamma$ contains a screw motion $\gamma$, where $\gamma$ is given by some rotation followed by translation by some vector $v \in \R^3$. We may assume that the angle of rotation is an irrational multiple of $2\pi$, for otherwise, applying $\gamma$ some number of $k$ times yields a (pure) translation, and we are in the first case again.

Denote by $P$ the plane through the origin orthogonal to $v$, and for $t \in \R$ let \hbox{$P_t := P + tv$}, the parallel translate of $P$ by the vector $tv$. Consider the fiber $\ell_0$ through the origin, and let $\ell_t := \ell_{tv}$. We need the following additional lemma.
\begin{lemma}\label{lem:helplem}
Either $\ell_t \subset P_t$ for all $t$, or $\ell_0$ is parallel to $v$.
\end{lemma}
\begin{proof}
The statement is equivalent to saying that if $\ell_t$ is transverse to $P_t$ for some $t$, then $\ell_t$ is parallel to $v$. Therefore, for the sake of contradiction, let us assume that there is some $t \in \R$ such that $\ell := \ell_t$ is transverse to $P_t$ (and hence transverse to $P$) and not parallel to $v$. For simplicity assume that $t = 0$. Let $\pi \col \E^3 \to P$ be the orthogonal projection onto $P$. Then $\ell$ projects to a line $\pi(\ell) \subset P$. Now, consider the projected lines $\pi(\ell_t)$ for $t \in \R$. If $\pi(\ell_t) = \pi(\ell)$ for all $t \in \R$, then the lines $\ell_t$ must be pairwise parallel, thus the plane $Q$ spanned by $v$ and $\ell$ is fibered by (parallel) lines. Then $\gamma$ must preserve $Q$ in order for the fibration $\lfol$ to be preserved, which is only possible if $\gamma$ is trivial, a contradiction. Hence, we may assume that there is some $t_0 \in \R$ such that $\pi (\ell_{t_0}) \neq \pi(\ell)$. We may further assume (without loss of generality) that $t_0 < 0$ and that every $\ell_t$, for $t \in [t_0,0]$, intersects $P$ transversely (by choosing $t_0$ close enough to $0$). Now let 
\[
N := \bigcup_{t \in [t_0,0]} \ell_t \subset \E^3.
\]
Then the projection $\pi(N) \subset P$ contains the cone $K \subset P$ given by the convex hull of $\pi(\ell)$ and $\pi(\ell_{t_0})$ (see Figure \ref{fig:Lemma}). Let $\theta$ be the angle between $\pi(\ell)$ and $\pi(\ell_{t_0})$. Since the angle of rotation of $\gamma$ is irrational, there is some $k \in \N$ such that the projection of $\gamma^k(\ell) \in \mc{F}$ onto $P$ is a line obtained by rotating $-\pi(\ell)$ towards the interior of $K$ by an angle of less then $\theta$. In other words, $\pi(\gamma^k(\ell)) \subset \intr K$. From this we deduce that $\gamma^k(\ell)$ intersects $N$. However, since $k > 0 > t_0$, we see that $\gamma^k(\ell) \subsetneq N$, hence $\gamma^k(\ell)$ intersects some line in $N$ transversely, a contradiction.  
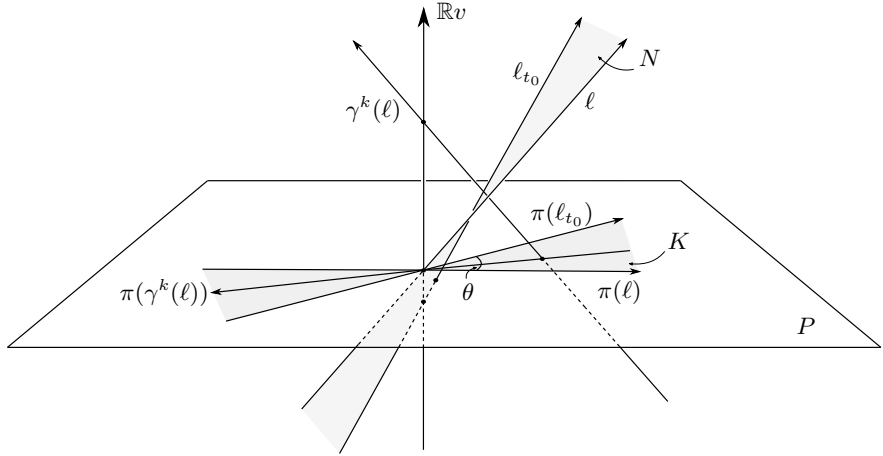
\begin{figure}[ht]
\labellist
\small\hair 2pt
\pinlabel $N$ [l] at 394 260
\pinlabel $\ell$ [l] at 360 233
\pinlabel $\ell_{t_0}$ [r] at 340 250
\pinlabel $\gamma^k(\ell)$ [l] at 210 228
\pinlabel $\R v$ [l] at 267 290
\pinlabel $\theta$ [l] at 283 113
\pinlabel $K$ [l] at 412 145
\pinlabel $P$ [l] at 493 92
\pinlabel $\pi(\ell)$ [l] at 368 113
\pinlabel $\pi(\ell_{t_0})$ [l] at 326 162
\pinlabel $\pi(\gamma^k(\ell))$ [l] at 68 113
\endlabellist
\centering 
\includegraphics[scale=0.6]{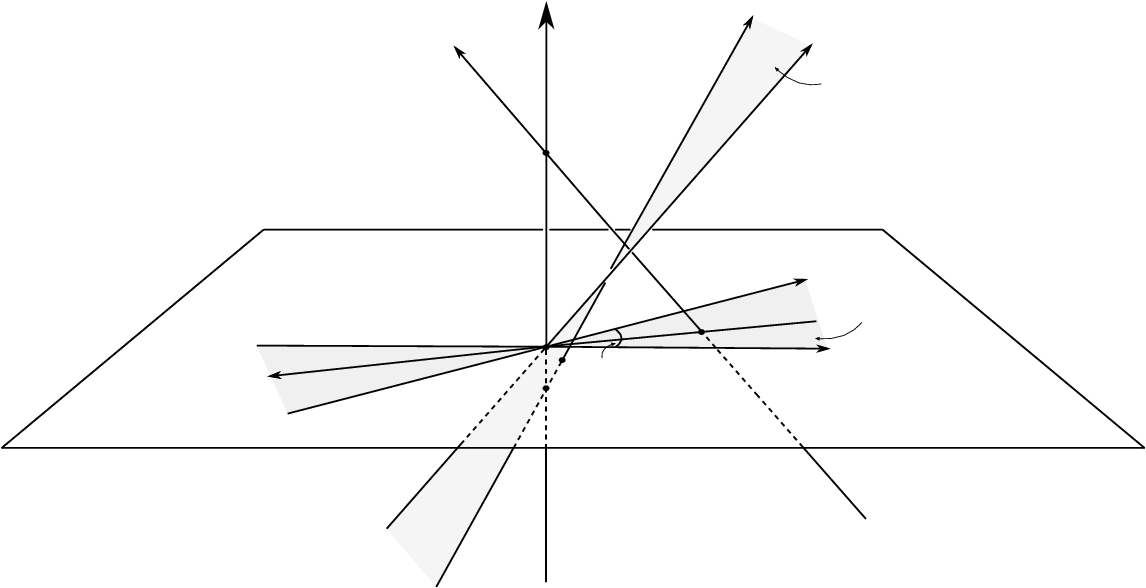}
\caption{The line $\gamma^k(\ell)$ intersects $N$ transversely.}
\label{fig:Lemma}
\end{figure}
\end{proof}
\begin{proof}[Proof of Theorem \ref{thm:main1} (cont.)]
Using Lemma \ref{lem:helplem}, we now have to consider two cases. The first is that $\ell_t \subset P_t$ for all $t \in \R$. Let us show that, under this assumption, every fiber of $\lfol$ must be contained in one of the planes $P_t$ (in particular, the fibration will be one-parameter). To see this, note first that each of the oriented lines $\ell_t$ divides $P_t$ into two open, oriented half-planes $\ell_t^+$ and $\ell_t^-$, where $\partial \ell_t^+ = \ell_t$ and $\partial \ell_t^- = -\ell_t$ (that is, $\ell_t$ with the opposite orientation). 
Here, the orientations of $\ell_t^+$ and $\ell_t^-$ come from a consistently chosen orientation of the $P_t$. Now assume that there is a point $p \in P$ such that $\ell_p$ intersects $P$ (and hence every $P_t$) transversely. We may assume that $\ell_p$ is not parallel to $v$, so that the orthogonal projection $\pi(\ell_p) \subset P$ of $\ell_p$ is a line again. Furthermore, we may assume that $\pi(\ell_p)$ intersects $\ell_0$ transversely (if that is not the case, simply replace $P$ by an appropriate $P_t$, and $\ell_0$ by $\ell_t$, for some $t \in \R$). Now, without loss of generality, let us assume that $p \in \ell_0^+$. Then the point $p_t$ given by the intersection of $\ell_p$ with $P_t$ must be contained in $\ell_t^+$ for every $t \in \R$, for otherwise, the line $\ell_p$ intersects one of the $\ell_t$ transversely. But since $\pi(\ell_p)$ intersects $\ell_0$ transversely, there is some $T > 0$ such that $\pi (p_t) \in \ell_0^-$ for all $t > T$. Again, since the angle of rotation of $\gamma$ is irrational, we can approximate $\ell_0$ arbitrarily well by $\pi(\gamma^k(\ell_0)) = \pi(\ell_k)$ for large enough $k \in \N$, hence we can approximate $\ell_0^-$ by $\pi(\gamma^k(\ell_0))^-$. In particular, there is some $k \geq T$ such that $\pi(p_k) \in \ell_0^- \cap \pi(\ell_k)^-$. But then $p_k \in \ell_k^-$, a contradiction.

The other case is that $\ell_0$ is parallel to $v$ (and then, in particular, $\ell_t = \ell_0$ for all $t$). We will show that in this case, every fiber must be parallel to $v$, and so the fibration is trivially one-parameter. Arguing again by contradiction, we assume that there are fibers not parallel to $v$. Choose a small closed disc $D \subset P$ such that
\begin{itemize}
    \item[(i)] for every $p \in D$, the fiber $\ell_p$ is transverse to $D$;
    \item[(ii)] for every $p \in \partial D$, the fiber $\ell_p$ is not parallel to $v$.
\end{itemize}
Such a disc can be found as follows. First, take a disc $D = D_r(0)$ (the closed disc about $0$ of radius $r > 0$) that satisfies (i). Now if (ii) does not hold, then there is some $p_0 \in \partial D$ such that $\ell_{p_0}$ is parallel to $v$. By applying $\gamma$ successively (once again using the fact that its rotational angle is irrational) we find that for a dense subset of $\partial D$, the corresponding fibers must be parallel to $v$. Then by continuity, this must hold for every fiber through points in $\partial D$. But then the set of all fibers through $\partial D$ form a straight cylinder parallel to $v$, and thus every fiber inside that cylinder must be parallel to $v$ as well. In other words, the fibration is constant over $D$. But since the fibration is assumed to be globally non-constant, we find a larger disc, again called $D$, so that (i) is still satisfied and the fibration is not constant over $D$. Then $D$ has to satisfy (ii) as well.

Now let $\Sigma := \{\ell_p \col p \in \partial D\}$ be the surface consisting of all fibers through points in $\partial D$. Let $\Sigma_t := \Sigma \cap P_t$, with $P_t = P + tv$ as before, and let $\pi(\Sigma_t)$ be its projection to $P$. 
We shall prove that there is some $T>0$ such that for all $t \in \R$ with $t > T$ or $t < -T$ we have that
\begin{equation}\label{eq:sigma}
D \subset \intr \pi(\Sigma_t),
\end{equation}
where $\intr \pi(\Sigma_t)$ denotes the interior of $\pi(\Sigma_t)$, that is, the connected component of $P_t \setminus \pi(\Sigma_t)$ bounded by $\pi(\Sigma_t)$ with compact closure.
Indeed, $\pi(\Sigma_t)$ is obtained from $\pi(\Sigma_0) = \partial D$ by flowing in the direction of the projected lines $\pi(\ell_p), \, p \in \partial D$. Denote this flow by $\Phi_{\tau}$. For $T$ large enough and $t > T$ or $t < -T$, the set $\Phi_{t}(\partial D)$ lies outside of $D$, that is, $\Phi_t(\partial D) \subset P \setminus D$. The fact that we can write $D$ instead of $\intr D$ here is because the $\ell_p$ project to lines and not points, due to property (ii) above; hence no point on $\partial D$ is fixed under the flow $\Phi$. Since none of the projected lines point to the origin (due to Lemma \ref{lem:helplem}), the origin stays in the interior while applying the flow, from which (\ref{eq:sigma}) follows. This is illustrated in Figure \ref{fig:circle}. 

\begin{figure}[ht]
\centering 
\labellist
\small\hair 2pt
\pinlabel $D$ [l] at 180 227
\pinlabel $\pi(\Sigma_t)$ [l] at 220 180
\endlabellist
\centering
\includegraphics[scale=0.55]{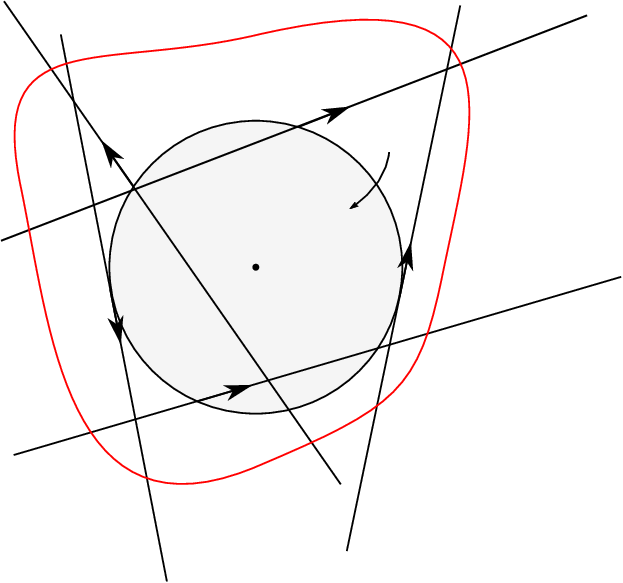}
\caption{$D$ is contained in the interior of $\pi(\Sigma_t)$.}
\label{fig:circle}
\end{figure}

Now let $k > T$ and consider the surface $\tilde{\Sigma} := \gamma^k(\Sigma)$. Then, since $\lfol$ is invariant under the action of $\Gamma$, we see that $\tilde{\Sigma}$, too, is a union of fibers of $\lfol$. Hence, either $\Sigma$ and $\tilde{\Sigma}$ are disjoint, or they intersect in a set of common fibers. In particular, the intersection $\Sigma \cap \tilde{\Sigma}$ is either empty or there is a non-empty intersection in every $t$-level, that is, $\Sigma_t \cap \tilde{\Sigma}_t \neq \emptyset$ for every $t$. On the other hand, from (\ref{eq:sigma}) we deduce that $\pi(\tilde{\Sigma}_k) = \partial D \subset \intr \pi(\Sigma_k)$, hence $\tilde{\Sigma}_k \subset \intr \Sigma_k$. Similarly, one shows that $\Sigma_0 \subset \tilde{\Sigma}_0$, see Figure \ref{fig:hyperboloid} below. But this means that $\Sigma \cap \tilde{\Sigma} \neq \emptyset$ while $\Sigma_k \cap \tilde{\Sigma}_k = \emptyset$, a contradiction. \qedhere

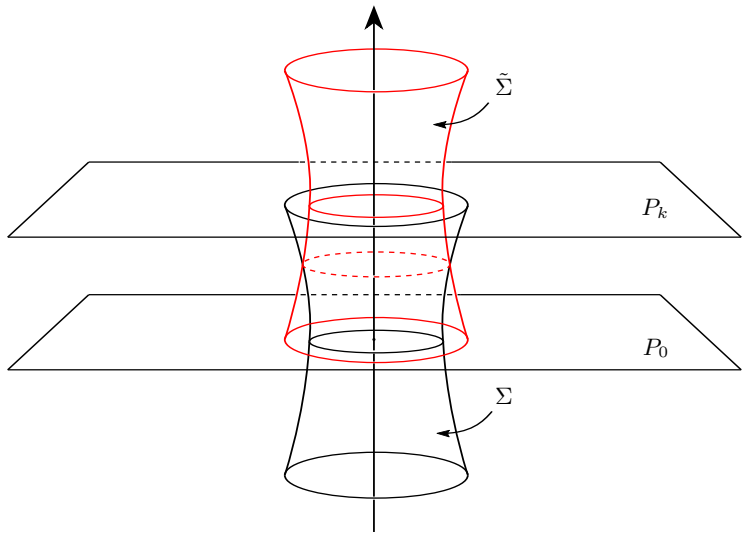
\begin{figure}[ht]
    \labellist
    \small\hair 2pt
    \pinlabel $\Sigma$ [bl] at 248 75
    \pinlabel $\tilde{\Sigma}$ [bl] at 248 230
    \pinlabel $P_k$ [r] at 340 175
    \pinlabel $P_0$ [r] at 340 105
    \endlabellist
    \centering
    \includegraphics[scale=0.75]{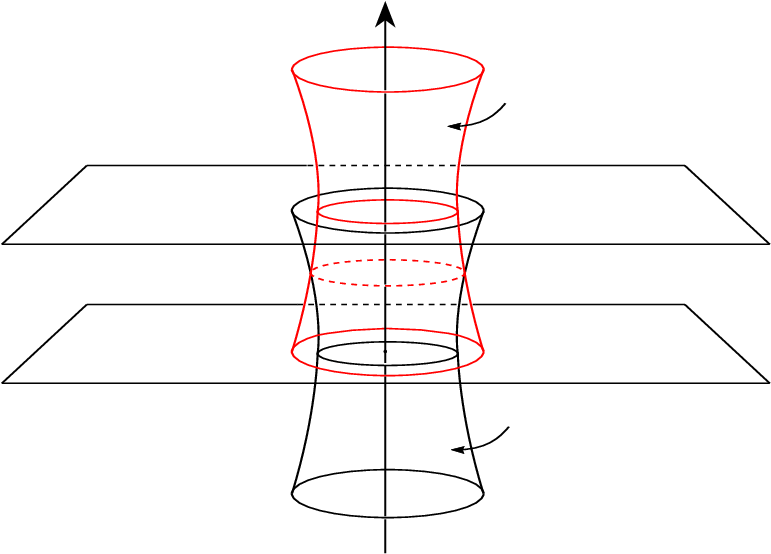}
    \caption{The surfaces $\Sigma$ and $\tilde{\Sigma}$ intersect transversely.}
    \label{fig:hyperboloid}
\end{figure}

\end{proof}

\section{The volume of a geodesible vector field}\label{section:volume}
In this section, we introduce the notion of volume for a geodesible vector field, following \cite{geiges:vector}. Let $X$ be a non-vanishing vector field on a manifold $M$. By a result of Wadsley \cite{wadsley:1975}, $X$ is geodesible if and only if there exists a 1-form $\alpha$ such that $\alpha(X) = 1$ and $i_X\drm \alpha = 0$. This 1-form is called \emph{connection form} (sometimes also \emph{characteristic form}) for $X$. Note that, in particular, this characterization implies the geodesibility of Reeb vector fields. Now, assume that $M$ is closed and $\dim M = 3$. Define the \emph{volume} of $X$ as
\[
\vol_X := \int_M \alpha \wedge \drm \alpha.
\]
This does not depend on the specific choice of the connection form $\alpha$, as follows from the following identity for arbitrary 1-forms $\alpha$ and $\beta$:
\begin{equation}\label{eq:oneforms}
\alpha \wedge \drm \alpha - \beta \wedge \drm \beta = (\alpha - \beta) \wedge (\drm \alpha + \drm \beta) + \drm (\alpha \wedge \beta).
\end{equation}
Namely, if $\alpha$ and $\beta$ are connection 1-forms for $X$, then $(\alpha-\beta)\wedge (\drm \alpha + \drm \beta) = 0$, thus, (\ref{eq:oneforms}) implies that
\[
\int_M \alpha \wedge \drm \alpha - \int_M \beta \wedge \drm \beta = \int_M \drm (\alpha \wedge \beta) = 0,
\]
where the last equality follows from the assumption of $M$ being closed. Similarly, one can define the volume of geodesible vector fields on higher-dimensional manifolds. I refer the reader to \cite{geiges:vector} for more details, also regarding the computation of $\vol_X$.

Now, let $p \col M \to N$ be a $k$-fold covering and $X$ a geodesible vector field on $N$. Let $\alpha$ be a connection form for $X$. Then the 1-form $p^* \alpha$ is clearly a connection form for the lifted vector field $Y$ (in particular, $Y$ is geodesible). The volumes of $X$ and $Y$ are related as
\begin{equation}\label{eq:volumecovering}
\vol_Y = \int_M p^*(\alpha \wedge \drm \alpha) = k \int_N \alpha \wedge \drm \alpha = k \, \vol_X.
\end{equation}

Formula (\ref{eq:oneforms}) can also be used to derive a slight generalization of Proposition 2.1 in \cite{geiges:vector}, see Lemma \ref{lem:samereeb} below. Before formulating the statement, let us introduce some notation that will also be used throughout the remainder of this paper. Given two vector fields $X$ and $Y$ on a manifold $M$, we write $X \sim Y$ if there is a positive function $\lambda \in C^{\infty} (M)$ such that $Y = \lambda X$.

\begin{lemma}\label{lem:samereeb}
Let $\alpha_0$ and $\alpha_1$ be two contact forms on a closed 3-manifold $M$ such that $R_{\alpha_0} \sim R_{\alpha_1}$. Then, the contact structures $\ker \alpha_0$ and $\ker \alpha_1$ are diffeomorphic.
\end{lemma}
\begin{proof}
Note that it suffices to prove that $\alpha_0 \wedge \drm \alpha_0$ and $\alpha_1 \wedge \drm \alpha_1$ define the same orientation of $M$. Indeed, if this is the case, the 2-forms $\drm \alpha_0$ and $\drm \alpha_1$ restrict to non-degenerate forms defining the same orientation on any hyperplane field transverse to $R_{\alpha_0} \sim R_{\alpha_1}$. Then, the 1-form $\alpha_t := (1-t) \, \alpha_0 + t \, \alpha_1$ is contact for every $t \in [0,1]$, and the statement follows from Gray stability (see \cite[Thm. 2.2.2]{geiges:2008}). 

Therefore, for the sake of contradiction, assume that the orientations induced by $\alpha_0$ and $\alpha_1$ are opposite. Since $R_{\alpha_0} \sim R_{\alpha_1}$, the 2-forms $\drm \alpha_0$ and $\drm \alpha_1$ must be multiples of each other, so we may write $\drm \alpha_1 = \mu \, \drm \alpha_0$ where $\mu$ is a function $M \to \R_{>0}$. Also, set $\lambda := \alpha_1(R_{\alpha_0}) \in C^{\infty}(M,\,\R_{>0})$.
It follows that
\[
\alpha_1 \wedge \drm \alpha_1 = \lambda \mu \, \alpha_0 \wedge \drm \alpha_0,
\]
and 
\[
(\alpha_0 - \alpha_1) \wedge (\drm \alpha_0 + \drm \alpha_1) = (1 - \lambda) (1 + \mu) \, \alpha_0 \wedge \drm \alpha_0.
\]
Then, identity (\ref{eq:oneforms}) implies that
\begin{align*}
\int_M (1 - \lambda \mu) \, \alpha_0 \wedge \drm \alpha_0 &= \int_M \alpha_0 \wedge \drm \alpha_0 - \int_M \alpha_1 \wedge \drm \alpha_1 
\\&= \int_M (\alpha_0 - \alpha_1) \wedge (\drm \alpha_0 + \drm \alpha_1)
\\&= \int_M (1- \lambda)(1+\mu) \, \alpha_0 \wedge \drm \alpha_0.
\end{align*}
But then $\int_M (\mu - \lambda) \, \alpha_0 \wedge \drm \alpha_0$ must vanish, which is impossible since $\mu - \lambda$ is assumed to be negative everywhere. Hence, we arrive at a contradiction.
\end{proof}

\section{Proof of Theorem \ref{thm:main2}}\label{section:main}
Let $M$ be a closed orientable complete flat 3-manifold. As discussed in the introduction, $M$ can be written (up to affine diffeomorphism) as $M = T^3 / \Gamma$, where $\Gamma < \isom(T^3)$ is a finite subgroup of isometries of $T^3$ acting freely and orientation-preservingly. Here, $T^3$ is the standard flat 3-torus, that is, $T^3 = \E^3 / (2 \pi \, \Z^3)$, where the $\Z^3$-action is generated by translations in the standard coordinate directions. By the following proposition, it is in fact enough to consider $M = T^3$.

\begin{prop}
Let $X$ be a geodesic vector field on $M = T^3 /\Gamma$ and $\xlif$ its lift to $T^3$. Then $X$ is conformally Reeb if and only if $\xlif$ is conformally Reeb.
\end{prop}
\begin{proof}
Assume first that $X$ is conformally Reeb. That is, there is a contact form $\alpha$ on $M$ such that $X \sim R_{\alpha}$. Let $\pi \col T^3 \to M$ be the natural projection. Then $p^* \alpha$ is again a contact form, and clearly $R_{p^* \alpha} \sim \xlif$.

Conversely, assume that $\xlif \sim R_{\alphtil}$ for some contact form $\alphtil$ on $T^3$. Since $\abs{\Gamma} < \infty$, we can average under the action of $\Gamma$ to get a 1-form
\[
\alpha := \frac{1}{\abs{\Gamma}} \sum_{\gamma \in \Gamma} \gamma^* \alphtil.
\]
Then $\alpha$ is again a contact form, since $\gamma_* \xlif = \xlif$ and $\drm \gamma$ maps the hyperplane field $\xlif^{\perp}$ orientation-preservingly to itself, for every $\gamma \in \Gamma$. Here we are using the fact that in dimension 3, a 1-form $\beta$ with $\beta(\xlif) > 0$ is contact if and only if $\drm \beta$ is non-vanishing on any hyperplane field transverse to $\xlif$.  It also follows that $R_{\alpha} = R_{\alphtil} \sim \xlif$. Now since $\gamma^* \alpha = \alpha$ for every $\gamma \in \Gamma$, $\alpha$ descends to a contact form on $M$, whose Reeb vector field is a multiple of $X$.
\end{proof}

We may now, for the remainder of the section, assume that $M = T^3$. We may further assume that the geodesic vector field $X$ is not constant, for otherwise, there is an embedded 2-torus transverse to $X$ and so by Stokes' theorem, $X$ cannot be (conformally) Reeb.
Let $\xdlif$ be the lift of the geodesic vector field $X$ to $\E^3$. By Theorem \ref{thm:main1}, $\xdlif$ is tangent to a fibration $\mc{P}$ of affine planes. Now choose a parallel orthonormal frame $(E_1, E_2, E_3)$ of $\E^3$ such that $E_1$ and $E_2$ span the fibers of $\mc{P}$. This frame descends to an orthonormal frame of $T^3$, which we call $(E_1,E_2,E_3)$ again. Then $E_1$ and $E_2$ span the leaves of the totally geodesic foliation $\mc{P}_T$ of $T^3$ covered by $\mc{P}$. Let us see that $\mc{P}_T$ is in fact a $T^2$-fibration over $S^1$. First note that the leaves are embedded copies of $T^2$. Indeed, each leaf $P_T \in \mc{P}_T$ is covered by a plane $P \in \mc{P}$, hence $P_T$ is either a 2-torus, or a dense immersed cylinder $S^1 \times \R$, or a dense immersed copy of $\R^2$. But since $X$ is constant on each leaf, the existence of dense leaves would force $\xlif$ to be globally constant, which we already ruled out. 

Next, define a map $\zeta \col T^3 \to S^1$ as follows. Fix a 2-torus $P_T \in \mc{P}_T$ covered by a plane $P \in \mc{P}$. Then, the orbit of $P$ under the action of $2 \pi \, \Z^3$ is a discrete set of equidistant planes. This follows from the fact that $P_T$ is not dense in $T^3$, and that the $\Z^3$-action on $\E^3$ commutes with the action of $\R^3$ by translations. For the same reason, the minimal distance $t_0$ of two such planes does not depend on the specific fiber $P_T$ of $\mc{P}_T$. Then, for $q \in T^3$, we may define $\zeta(q)$ as
\[
\zeta(q) := \frac{2\pi \, t_q}{t_0} \mod 2\pi \in S^1 = \R / 2\pi  \Z,
\]
where $t_q > 0$ is the smallest number such that $\Phi_{t_q}(q)$ (where $\Phi$ is the flow of $E_3$) lies in the fiber $P_T$. This defines a fibration of $T^3$ whose fibers are the elements of $\mc{P}_T$. Now we can write $X$ as
\begin{equation}\label{eq:xlif}
X = \sin \theta(\zeta) \, E_1 + \cos \theta(\zeta) \, E_2
\end{equation}
for some function $\theta \col S^1 \to S^1$. Using the identification $S^1 = \R / 2\pi \Z$, we may think of $\theta$ (or any other function $S^1 \to S^1$) as a function $\R \to \R$, such that $\theta(t+2\pi) - \theta(t) \in 2 \pi \, \Z$ for all $t \in \R$. In fact, we have that $\theta(t+2\pi) - \theta(t) = 2\pi \, \deg \theta$, where $\deg \theta$ is the degree of $\theta$ as a map of $S^1$. By $\theta^\prime$ we mean the usual derivative of $\theta$ when viewed as a function defined on $\R$.

We will first prove the following proposition.
\begin{prop}\label{lem:mainLemma}
Let $X$ be a geodesic vector field on $T^3$. Then the following are equivalent.
\begin{enumerate}[(i)]
    \item X is conformally Reeb.
    \item $\deg \theta \neq 0$ and for any $a, b \in \R, \, a < b$ we have that
    \[
    \theta(b) - \theta(a) > - \pi, \quad \text{if } \deg \theta > 0,
    \]
    and
    \[
    \theta(b) - \theta(a) < \pi, \quad \text{if } \deg \theta < 0.
    \]
    \item The set
    \[
    \mc{B} := \left\{\varphi \in C^{\infty}(S^1,S^1) \col \varphi^\prime \neq 0, \, \abs{\varphi-\theta} < \frac{\pi}{2}\right\}
    \]
    is non-empty (here, $\abs{.}$ is the Euclidean norm (modulo $2\pi$)).
    
\end{enumerate}
\end{prop}
\begin{proof}
We first show that (iii) implies (ii). So assume that (iii) holds, and choose some $\varphi \in \mc{B}$. Note that $\deg \theta = \deg \varphi \neq 0$. If $\deg \theta > 0$, then $\varphi^\prime$ must be positive everywhere. Then, for $a < b$,
\[
\theta(a) - \frac{\pi}{2} < \varphi(a) < \varphi(b) < \theta(b) + \frac{\pi}{2},
\]
which implies that $\theta(b) - \theta(a) > - \pi$. A similar argument applies for the case of $\deg \theta$ being negative.

Conversely, if (ii) holds, we need to show that $\mc{B} \neq \emptyset$. We will do so by constructing some $\varphi \in \mc{B}$ explicitly. First, perhaps after applying a $C^0$-small perturbation, we may assume that $\theta$ has finitely many local minima and maxima, respectively, and no other critical points. That is, there is a subdivision
\[
0 < a_1 < b_1 < \ldots < a_n < b_n < 2\pi,
\]
such that $\theta$ has a local maximum at every $a_k$ and a local minimum at every $b_k$. Define intervals $I_k$ by
\[
I_k := [\theta(a_k)-\pi/2,\,\theta(b_k)+\pi/2].
\]
Note that it follows from (ii) that $I_k$ does indeed define an interval with non-empty interior. Moreover, we have that
\begin{equation}\label{eq:intervals}
\max I_k > \min I_l
\end{equation}
for every $k = 1,\ldots,n$ and every $l = 1, \ldots, k$.
Now we want to find some numbers 
\[
c_1 \leq c_2 \leq \ldots \leq c_n,
\]
such that $c_k \in I_k$ for every $k$. Given such numbers, we can define a function $\phi$ with the following properties:
\begin{itemize}
    \item $\abs{\phi - \theta} < \pi/2$ everywhere;
    \item $\phi$ is constantly equal to $c_k$ on $[a_k,b_k]$;
    \item $\phi$ is non-decreasing.
\end{itemize}
This is illustrated in Figure \ref{fig:phi}. This construction is possible since $\theta$ is strictly increasing on $(b_k,a_{k+1})$. Then $\phi$ can be approximated by a function in $\mc{B}$. Therefore, all we are left to do is find numbers $c_k$ as above. This is best done reversely, starting with $c_n$. Set $c_n := \max I_n$. The remaining $c_k$ are defined inductively as
\[
c_k := \min \set{c_{k+1}, \, \max I_k} \leq c_{k+1}.
\]
Note that $c_k \in I_k$ since $c_k$ is given by the maximum of some $I_l, \, l \geq k$, so that $c_k > \min I_k$ by (\ref{eq:intervals}). This concludes the proof of the equivalence of (ii) and (iii).

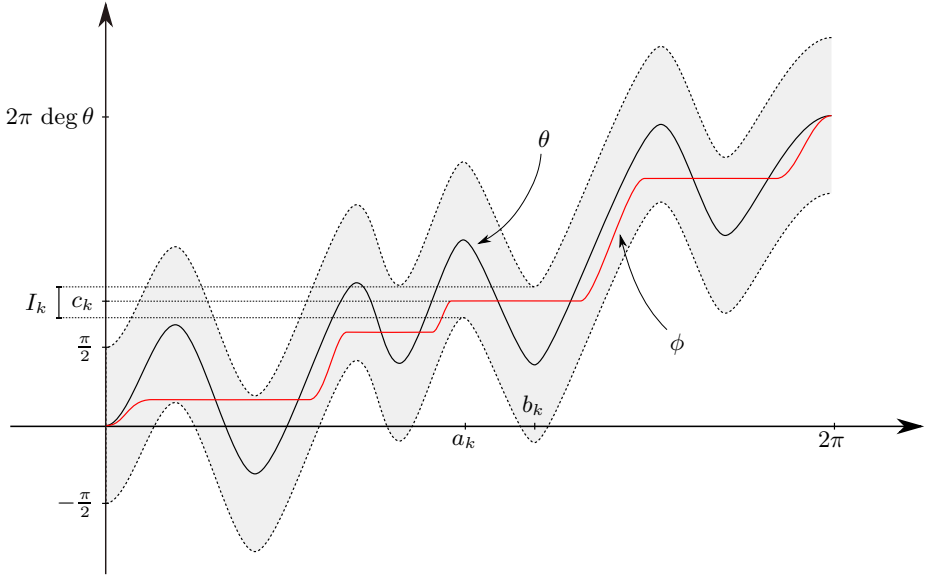
\begin{figure}[ht]
\labellist
\small\hair 2pt
\pinlabel $2\pi\,\deg\theta$ [r] at 85 350
\pinlabel $c_k$ [r] at 85 210
\pinlabel $I_k$ [r] at 50 210
\pinlabel $\frac{\pi}{2}$ [r] at 85 176
\pinlabel $-\frac{\pi}{2}$ [r] at 85 58
\pinlabel $a_k$ [t] at 360 115
\pinlabel $b_k$ [b] at 412 122
\pinlabel $2\pi$ [t] at 637 117
\pinlabel $\theta$ [bl] at 413 325
\pinlabel $\phi$ [tl] at 512 190
\endlabellist
\centering
\includegraphics[scale=0.5]{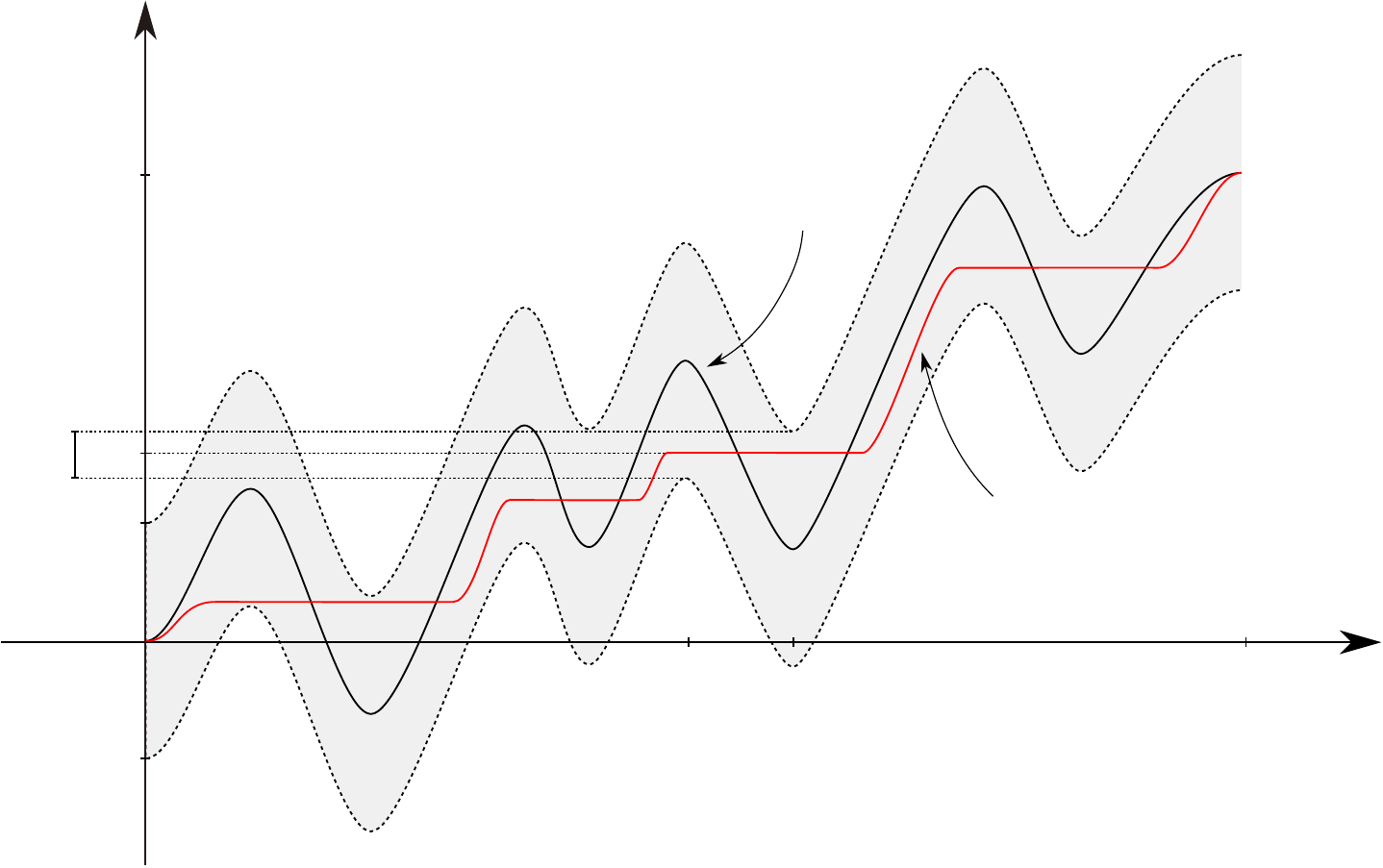}
\caption{Construction of a function $\phi$ that can be approximated by a function in $\mc{B}$.}
\label{fig:phi}
\end{figure}


Now, let us see how (iii) implies (i). Given $\varphi \in \mc{B}$ as in (iii), consider the 1-form
\[
\alpha = \sin \varphi(\zeta)\, \mc{E}^1 + \cos \varphi(\zeta)\, \mc{E}^2,
\]
where $(\mc{E}^1,\mc{E}^2,\mc{E}^2)$ is the dual frame to $(E_1,E_2,E_3)$. A simple calculation shows that 
\[
\drm \alpha = \varphi^\prime(\zeta) \cos \varphi(\zeta) \, \mc{E}^3 \wedge \mc{E}^1 - \varphi^\prime(\zeta) \sin \varphi(\zeta) \, \mc{E}^3 \wedge \mc{E}^2,
\]
hence
\[
\alpha \wedge \drm \alpha = \varphi^\prime(\zeta) \, \mc{E}^1 \wedge \mc{E}^2 \wedge \mc{E}^3 \neq 0,
\]
so $\alpha$ is a contact form. Its Reeb vector field is given by 
\[
R_{\alpha} = \sin \varphi(\zeta) \, E_1 + \cos \varphi(\zeta) \, E_2.
\]
We claim that, for a suitably chosen $\varphi \in \mc{B}$, there are functions $f, g \in C^{\infty}(S^1, \R_{>0})$, such that
\begin{equation}\label{eq:reeb} 
f(\zeta) \, X = R_{(1/g(\zeta))\alpha}.
\end{equation}
Generally, for $h \in C^{\infty}(T^3, \R_{>0})$, we have that $R_{(1/h)\alpha} = h \, R_{\alpha} + Y$, where $Y$ is the unique vector field satisfying $\alpha(Y) = 0$ and
\begin{equation}\label{eq:y}
i_Y \drm \alpha = \drm h(R_{\alpha}) \, \alpha - \drm h.
\end{equation}
Now, if $h = g \circ \zeta$ for some function $g \in C^{\infty}(S^1, \R_{>0})$, then $\drm h(R_{\alpha}) = \drm g \circ \drm \zeta (R_{\alpha}) = 0$, so (\ref{eq:y}) translates to
\[
i_Y \drm \alpha = -\drm h = - \drm g \circ \drm \zeta = - g^\prime(\zeta) \, \mc{E}^3,
\]
where we again think of $g$ as a $2\pi$-periodic function $\R \to \R$, with $g^\prime$ being its usual derivative.
Then, to solve equation (\ref{eq:reeb}), we need to find functions $f$ and $g$ such that $Y :=  f(\zeta) \, X - g(\zeta) \, R_{\alpha}$ satisfies
\begin{equation}\label{eq:eq1}
0 = \alpha(Y) = f(\zeta) \, \alpha(X) -g(\zeta),
\end{equation}
as well as
\begin{equation}\label{eq:eq2}
i_Y \drm \alpha = - g^\prime(\zeta)\, \mc{E}^3.
\end{equation}
Now (\ref{eq:eq1}) is equivalent to $g = f \cos(\varphi-\theta)$, which is positive iff $f$ is positive, since $\abs{\varphi-\theta} < \pi/2$. 
Then (\ref{eq:eq2}) translates to
\[
f\, \varphi^\prime  \sin(\varphi - \theta) \, \mc{E}^3 = -(f^\prime  \cos(\varphi-\theta) - f \, (\varphi^\prime - \theta^\prime) \sin(\varphi-\theta))\,\mc{E}^3,
\]
where we refrained from writing $\zeta$ in the arguments for simplicity. This, in turn, reduces to
\[
f^\prime  \cos(\varphi - \theta) + f \, \theta^\prime  \sin(\varphi-\theta) = 0. 
\]
This differential equation is being solved by
\[
f(x) := \exp\left(-\int_0^x \tan(\varphi(t) - \theta(t)) \, \theta^\prime(t) \drm t \right) > 0.
\]
However, for a generic choice of $\varphi$, $f$ is not $2\pi$-periodic, hence it does not define a function on $S^1$. Note that $f$ is $2\pi$-periodic if and only if
\[
I(\varphi) := \int_0^{2\pi} \tan(\varphi(t)-\theta(t)) \, \theta^\prime(t) \drm t
\]
vanishes.
Therefore, we need to show that the function $I \col \mc{B} \to \R$ has a zero. First observe that since 
\[
\int_0^{2\pi} \tan (\varphi(t) - \theta(t)) (\varphi^\prime(t) - \theta^\prime(t)) \,\drm t = \int_{x_0}^{x_0} \tan(u) \, \drm u = 0,
\]
we have that
\[
I(\varphi) = \int_0^{2\pi} \tan(\varphi(t) - \theta(t)) \, \varphi^\prime(t) \, \drm t.
\]
Now, it is easy to see that $\mc{B}$ is convex. Hence, it suffices to find functions $\varphi^+, \varphi^- \in \mc{B}$ such that $I(\varphi^+) > 0$ and $I(\varphi^-) < 0$. For then we can simply interpolate between $\varphi^+$ and $\varphi^-$ to find a zero of $I$.
To achieve this, one can adjust the construction of $\varphi$ in the proof of (ii) $\Rightarrow$ (iii) so that $\varphi < \theta$ wherever $\varphi$ is not constant (in fact, the function $\varphi$ drawn in Figure \ref{fig:phi} has this property). By approximating this function with a function in $\mc{B}$, we obtain a function $\varphi^- \in \mc{B}$ with $I(\varphi^-) < 0$. The function $\varphi^+$ is constructed similarly.

To finish the proof, we show that (i) implies (ii). Assume that $X \sim R_{\alpha}$ for some contact form $\alpha$ of $T^3$. Suppose, for the sake of contradiction, that (ii) does not hold. Assume for the moment that $\deg \theta > 0$. Then (ii) being false means that there are $a, b \in [0,2\pi]$ with $a < b$ such that $\theta(b) - \theta(a) = - \pi$, as well as $c, d \in [0,2\pi]$ with $b < c < d$ such that $\theta(c) = \theta(b)$ and $\theta(d) = \theta(a)$ (since $\deg \theta >0$). Furthermore, we may choose $a,b$ and $c,d$ so that 
\begin{equation}\label{eq:theta}
\theta(x) \in [\theta(b),\theta(a)] \quad \text{for all} \quad x \in [a,b] \cup [c,d].
\end{equation}
Now choose a point $p \in \E^3$ that projects to a point in $\zeta^{-1}(a) \subset T^3$ and let $P$ be the affine plane in $\E^3$ through $p$ spanned by $E_3$ and $\xdlif(p)$. Assume for the moment that $P$ covers a 2-torus in $T^3$, which we call $\Sigma$. Consider the two subsets
\[
\Sigma_1 := \Sigma \cap \{a \leq \zeta \leq b\}, \quad \Sigma_2 := \Sigma \cap \{c \leq \zeta \leq d\}.
\]
Both $\Sigma_1$ and $\Sigma_2$ are diffeomorphic to cylinders, and their boundaries are integral curves of $X$. Choose any orientation for $\Sigma$ and orient $\Sigma_1$ and $\Sigma_2$ accordingly as submanifolds of $\Sigma$. Denote the (oriented) boundary curves of $\Sigma_1$ and $\Sigma_2$ by
\[
\partial \Sigma_1 = \gamma_a \sqcup \gamma_b, \quad \partial \Sigma_2 = \gamma_c \sqcup \gamma_d.
\]
We may choose the orientation of $\Sigma$ so that $\gamma_a$ and $\gamma_b$ are negatively tangent to $X$, whereas $\gamma_c$ and $\gamma_d$ are positively tangent, see Figure \ref{fig:sigma}.
It follows that
\[
\int_{\Sigma_1} \drm \alpha = \int_{\gamma_a} \alpha + \int_{\gamma_b} \alpha < 0,
\]
and 
\[
\int_{\Sigma_2} \drm \alpha = \int_{\gamma_c} \alpha + \int_{\gamma_d} \alpha > 0.
\]
However, it follows from (\ref{eq:theta}) that $X$ (and then also $R_{\alpha})$ is positively transverse to the interiors of both $\Sigma_1$ and $\Sigma_2$. Then, since $\Sigma_1$ and $\Sigma_2$ are oriented consistently, $\int_{\Sigma_1} \drm \alpha$ and $\int_{\Sigma_2} \drm \alpha$ must have the same sign, and we arrive at a contradiction. 

We are left to deal with the case of $P$ covering some dense infinite cylinder in $T^3$ (instead of a 2-torus). Parametrize $P$ using coordinates $s$ and $t$, so that $\partial_s$ is identified with $\xdlif(p)$. Consider subsets of the form 
\[
P_{s_0} := P \cap \set{-s_0 \leq s \leq s_o} \subset P
\]
for some $s_0 > 0$. Then $P_{s_0}$ covers a cylinder in $T^3$, which we call $\Sigma = \Sigma^{s_0}$.
Defining $\Sigma_1 = \Sigma^{s_0}_1$ as before, we now have
\[
\partial \Sigma_1 = \partial_v \Sigma_1 \cup \partial_h \Sigma_1,
\]
where $\partial_v \Sigma_1 = \gamma_a \sqcup \gamma_b$ and $\partial_h \Sigma_1 = \partial \Sigma_1 \cap \partial \Sigma$. In other words, $\partial_v \Sigma_1$ and $\partial_h \Sigma_1$ are the "vertical" and "horizontal" part of $\partial \Sigma_1$, respectively. Note that, since $\alpha$ is non-zero on the vertical boundary components, we have that
\begin{equation}\label{eq:integral}
\left| \int_{\partial_v \Sigma^{t_0}_1} \alpha \right| > \left | \int_{\partial_v \Sigma^{s_0}_1} \alpha \right |
\end{equation}
for $t_0 > s_0$. Now let 
\[
C := \left| \int_{\partial_v \Sigma^{s_0}_1} \alpha \right|
\]
for some $s_0$, and choose $t_0 > s_0$ large enough so that 
\[
\left | \int_{\partial_h \Sigma^{t_0}_1} \alpha \right | < C.
\]
This can be done due to the fact that $P$ covers a dense cylinder in $T^3$.
Then (\ref{eq:integral}) implies that 
\[
\sgn \Bigg(\int_{\partial \Sigma^{t_0}_1} \alpha \Bigg) = \sgn \Bigg(\int_{\partial_v \Sigma^{t_0}_1} \alpha \Bigg),
\]
and the same may be assumed for $\Sigma^{t_0}_2$. Then, using the same reasoning as in the first case, we arrive at a contradiction again.

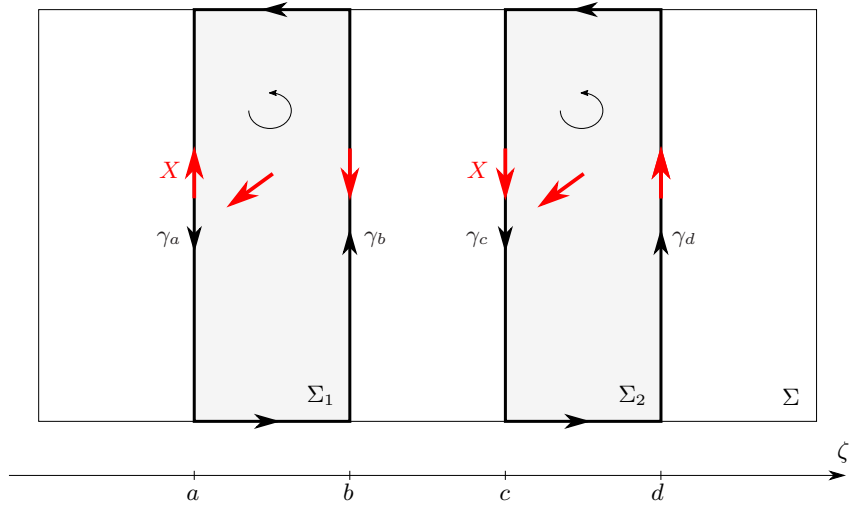
\begin{figure}[ht]
\labellist
\small\hair 2pt
\pinlabel $\Sigma_1$ [br] at 260 80
\pinlabel $\Sigma_2$ [br] at 495 80
\pinlabel $\Sigma$ [br] at 610 80
\pinlabel $a$ [b] at 150 7
\pinlabel $b$ [b] at 267 7
\pinlabel $c$ [b] at 385 7
\pinlabel $d$ [b] at 500 7
\pinlabel $\zeta$ [b] at 640 35
\pinlabel \color{red}$X$ [r] at 145 260
\pinlabel \color{red}$X$ [r] at 377 260
\pinlabel $\gamma_a$ [r] at 145 208
\pinlabel $\gamma_b$ [l] at 275 208
\pinlabel $\gamma_c$ [r] at 377 208
\pinlabel $\gamma_d$ [l] at 507 208
\endlabellist
\centering
\includegraphics[scale=0.5]{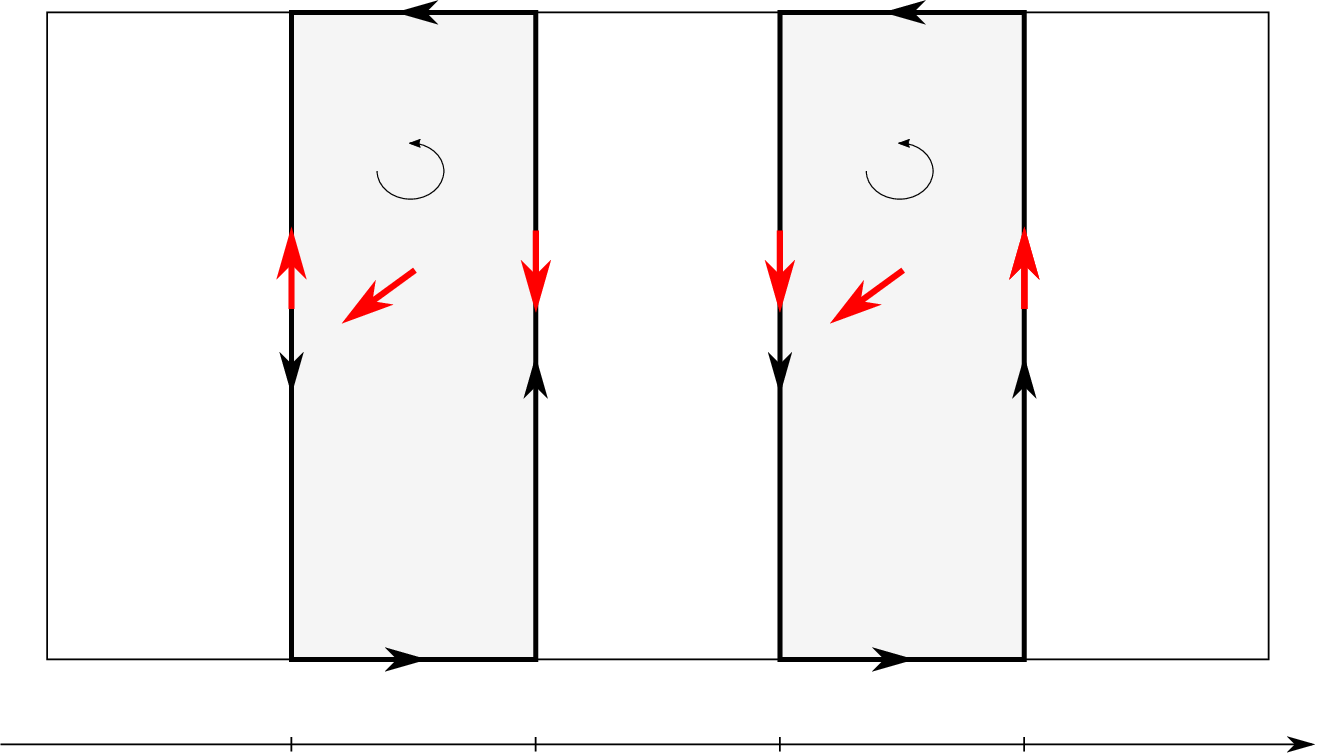}
\caption{$\Sigma_1$ and $\Sigma_2$.}
\label{fig:sigma}
\end{figure}

The case $\deg \theta < 0$ is analogous. Now we are still left to show that $\deg \theta$ is indeed non-zero. Note that if $\deg \theta = 0$ and the image of $\theta$ is contained in an open interval of length at most $\pi$, then there is an embedded 2-torus transverse to $X$, so that $X$ cannot be conformally Reeb. Therefore, we may again assume that there are $a < b < c < d$ with $\theta(b) - \theta(a) = \mp \pi$, $\theta(c) = \theta(b)$ and $\theta(d) = \theta(a)$, and we arrive at a contradiction using the same argument as before.
\qedhere
\end{proof}
\begin{proof}[Proof of Theorem \ref{thm:main2}]
Assume first that there is a geodesic vector field $Y$ on $M$ inducing a contact structure $\xi$ such that $X$ is everywhere transverse to $\xi$. In other words, $X$ and $Y$ are nowhere orthogonal. Then the same is true for the lifted vector fields $\xlif$ and $\ylif$ on $T^3$. In particular, both $\xlif$ and $\ylif$ are non-constant, so by the discussion prior to Proposition \ref{lem:mainLemma}, there are fibrations $\mc{T}_X$ and $\mc{T}_Y$ of $T^3$ by 2-tori tangent to $\xlif$ and $\ylif$, respectively. These two fibrations must coincide: Indeed, if this were not the case, we could consider a loop in some $T \in \mc{T}_X$ that is transverse to $\xlif$ and also transverse to $\mc{T}_Y$. Along this loop, $\xlif$ is constant, whereas $\ylif$ must make at least one complete turn (since $\ylif$ induces a contact structure), hence $\xlif$ and $\ylif$ are orthogonal somewhere, a contradiction. Therefore, writing $\xlif$ as
\[
\xlif = \sin \theta(\zeta) \, E_1 + \cos \theta(\zeta) \, E_2
\]
as in (\ref{eq:xlif}), we find that $\ylif$ is of the form
\begin{equation}\label{eq:yt}
\ylif = \sin \varphi(\zeta) \, E_1 + \cos \varphi(\zeta) \, E_2
\end{equation}
for some function $\varphi \col S^1 \to S^1$ with $\varphi^\prime \neq 0$. We also have that $|\varphi-\theta| < \pi/2$ everywhere. Thus, it follows from Proposition \ref{lem:mainLemma} that $X$ is conformally Reeb.

Conversely, assume that $X$ is conformally Reeb, that is, $X \sim R_{\alpha}$ for some contact form $\alpha$ on $M = T^3 / \Gamma$. Write the lifted vector field on $T^3$ again as $\xlif = \sin \theta(\zeta) \, E_1 + \cos \theta(\zeta) \, E_2$. Then, by Proposition \ref{lem:mainLemma}, there is a function $\varphi \col S^1 \to S^1$ such that $\varphi^\prime \neq 0$ and $|\varphi-\theta| < \pi/2$ everywhere. In particular, the geodesic vector field $\ylif := \sin \varphi(\zeta) \, E_1 + \cos \varphi(\zeta) \, E_2$ induces a contact structure and is nowhere orthogonal to $\xlif$. If $M = T^3$, then $Y = \ylif$ and we are done. So suppose that $M$ is not equal to $T^3$, that is, $M = T^3 / \Gamma$, where $\Gamma$ is a non-trivial subgroup of $\isom(T^3)$. We want to adjust the construction of $\ylif$ (resp. $\varphi$) so that it is invariant under the action of $\Gamma$, and therefore descends to a geodesic vector field $Y$ on $M$. First note that every element of $\Gamma$ must be a screw motion of finite order in $\Gamma$, since glide reflections are not orientation-preserving. If $\gamma \in \Gamma$ is such a screw motion, then $\gamma$ must preserve the fibration $\mc{T}$ of 2-tori defined by $\zeta$. Indeed, if there were some $T \in \mc{T}$ such that $\gamma(T) \notin \mc{T}$, then $\gamma(T)$ would intersect every fiber of $\mc{T}$ transversely. Now since $\xlif$ is constant along each fiber of $\mc{T}$ and also constant along $\gamma(T)$ (since $\gamma_* \xlif = \xlif$), it would follow that $\xlif$ is globally constant. In particular, $\xlif$ cannot be Reeb, a contradiction. But this means that the axis of rotation of $\gamma$ (and consequently its translational part) must be orthogonal to $\mc{T}$. In other words, the translation vector of $\gamma$ is a multiple of $E_3$. Now choose $\gamma_0 \in \Gamma$ so that the absolute value of its translational part is minimal among all elements of $\Gamma$. Then $\gamma_0$ generates $\Gamma$ (in particular, $\Gamma$ is cyclic). Write $\gamma_0$ as $\gamma_0 = T_{\lambda \, E_3} \circ R_{\phi}$,
where $R_{\phi}$ is rotation about the axis spanned by $E_3$ of angle $\phi$, and $T_{\lambda \, E_3}$ is translation by the vector $\lambda \, E_3$ for some real number $\lambda$. Then, it suffices to choose $\varphi$ so that $\varphi(t + \lambda) = \varphi(t) + \phi$ for all $t$, for then $\varphi \circ \zeta \circ \gamma_0 = \varphi \circ \zeta + \phi$ which implies that $(\gamma_0)_* \, Y_T = Y_T$. To find an appropriate $\varphi$, we can construct $\varphi$ first on the interval $[0,\lambda]$ as in the proof of Proposition \ref{lem:mainLemma}, and then extend it to the whole real line via $\varphi(t+\lambda) := \varphi(t) + \phi$. Of course one has to be a little careful regarding smoothness of $\varphi$, but this can be arranged easily. The vector field $\ylif$ we end up with is invariant under $\Gamma$, hence it descends to a vector field $Y$. 

To prove the second statement of the theorem, note that in Proposition \ref{lem:mainLemma} it is actually shown that if $\xlif$ is conformally Reeb for some contact form $\alpha_T$, then it is also conformally Reeb for a multiple of the contact form $\alpha_{\varphi} = \sin \varphi (\zeta) \, \mc{E}^1 + \cos \varphi(\zeta) \, \mc{E}^2$
whose kernel defines the contact structure $\xi_T$. Then, by Lemma \ref{lem:samereeb}, $\ker \alpha_T$ and $\xi_T$ are diffeomorphic. Now set 
\[
n := 2 \pi \, \deg \varphi = 2 \pi \, \deg \theta = \theta(2\pi) - \theta(0).
\]
Then $\alpha_{\varphi}$ pulls back to
\[
\alpha_n := \sin (n \, \zeta) \, \mc{E}^1 + \cos (n\, \zeta) \, \mc{E}^2
\]
via the diffeomorphism
\[
T^3 \longrightarrow T^3, \quad p \longmapsto (\Phi_{\varphi^{-1}(n \, \zeta(p))} \circ \Phi_{-\zeta(p)}) (p),
\]
where $\Phi$ denotes again the flow of $E_3$. On the other hand,
\begin{align*}
\abs{\Gamma} \, \vol_X = \vol_{\xlif} &= \int_{T^3} \alpha_T \wedge \drm \alpha_T 
\\&= \int_{T^3} \theta^\prime(\zeta) \, \mc{E}^1 \wedge \mc{E}^2 \wedge \mc{E}^3 
\\&= \theta(2\pi) \int_{\zeta^{-1}(2\pi)} \mc{E}^1 \wedge \mc{E}^2 - \theta(0) \int_{\zeta^{-1}(0)} \mc{E}^1 \wedge \mc{E}^2
\\&= n \, A,
\end{align*}
where the first equation follows from (\ref{eq:volumecovering}). Hence, $n = \abs{\Gamma} \, \vol_X /A$. \qedhere.
\end{proof}

\begin{proof}[Proof of Corollary \ref{cor:std}]
Choose global coordinates $(x,y,z)$ for $\R^3$ such that the frame $(E_1, E_2, E_3)$ on $T^3$ is covered by the coordinate frame $(\partial_x, \partial_y, \partial_z)$. Then, by Theorem \ref{thm:main2}, $\ker \tilde{\alpha}$ is diffeomorphic to the kernel of
\[
\tilde{\alpha}_n = \sin (n z) \, \dx + \cos(nz) \, \dy,
\]
which pulls back to $\alpha_{\text{st}} = \dz + x \, \dy$ via the diffeomorphism
\[
(x,y,z) \longmapsto \left(z \, \sin(ny) - \frac{x\,\cos(ny)}{n},z \, \cos(ny) + \frac{x\,\sin(ny)}{n},y\right). \qedhere
\]
\end{proof}

\section{Open flat 3-manifolds}\label{section:open}
We start with the following general result.
\begin{prop}\label{prop:reeb}
Let $M$ be an orientable $3$-manifold with $H_{dR}^2(M) = 0$, and $X$ a non-vanishing vector field on $M$ whose flow induces a free, proper $\R$-action. Then $X$ is the Reeb vector field of a contact form.
\end{prop}
\begin{proof}
Since $X$ induces a free and proper $\R$-action that is also orientation-preserving, the orbit space $B = M / \R$ is an orientable 2-dimensional manifold, and the projection $p \col M \to B$ defines a principal line bundle, which is necessarily trivial. That is, we can identify $M$ with $B \times \R$, where the $\R$-fibers correspond to the integral curves of $X$. Now $B$ is a deformation retract of $M$, so we have that $H_{dR}^2(B) = H_{dR}^2(M) = 0$. Hence, there is an exact area form $\omega = \drm \beta$ on $B$. Let $t$ denote the coordinate of the $\R$-fibers of $M = B \times \R$. Then the 1-form $\alpha := \drm t + p^* \beta$ is contact, and $R_{\alpha} = \partial_t = X$.
\end{proof}
\begin{cor}\label{cor:aper}
Let $X$ be an aperiodic geodesic vector field on $\E^3$ or $\R^2 \times S^1$. Then $X$ is the Reeb vector field of a contact form. \qed
\end{cor}
\begin{rmk} Of course, in the case of $\E^3$, every geodesic vector field is aperiodic; hence, Corollary \ref{cor:aper} implies that every geodesic vector field on $\E^3$ is Reeb.
\end{rmk}
The following two examples are to show that Theorem \ref{thm:main2} is not true in general for non-closed manifolds.
\begin{ex}\label{ex:open}
(i) Let $M$ be equal to $S^1 \times \R^2$ or $T^2 \times \R$ with coordinates $(x,y,z)$ and consider the geodesic vector field
\[
X = \sin \theta(z) \, \partial_x + \cos \theta(z) \, \partial_y,
\]
where $\theta \col \R \to \R$ is a smooth function defined as follows. Set $\theta(0) = \theta(2\pi) = 0$, $\theta(\pi) = - \pi$, and 
\[
\theta(z) \approx \begin{cases} -z, & 0 \leq z \leq \pi,
\\ z-2\pi, & \pi \leq z \leq 2\pi, \end{cases}
\]
where the approximation is $C^0$-close. Then extend $\theta$ to a $2\pi$-periodic function defined on $\R$. Since $\theta(\pi) - \theta(0) = -\pi$, the condition of Proposition \ref{lem:mainLemma} (or Theorem \ref{thm:main2}) is not satisfied. However, $X$ is still conformally Reeb. To see this, consider the 1-form $\beta = F(z) \, \dx + y \sin \theta(z) \, \dz$, where 
\[
F(z) := \int_0^z \cos \theta(t) \, \drm t -1.
\]
Then $\drm \beta = \cos \theta(z) \, \dz \wedge \dx + \sin \theta(z) \, \dy \wedge \dz$ is non-degenerate on the plane field $\eta$ spanned by $\partial_z$ and $\cos \theta(z) \, \partial_x - \sin \theta(z) \, \partial_y$, and $i_X \drm \beta = 0$. Furthermore,
\[
\beta(X) = F(x) \sin \theta(x) \approx \begin{cases} \sin x (1-\sin x) \geq 0, & \text{if }0 \leq x \leq \pi, \\ \sin x (\sin x - 1) \geq 0, & \text{if } \pi < x \leq 2 \pi. \end{cases}
\]
That is, $\beta(X) \geq - \varepsilon$ for some arbitrarily small $\varepsilon > 0$. Now choose $\varepsilon$ so that $1+2 \, \varepsilon \, \theta^\prime > 0$ everywhere, and consider the 1-form
\[
\alpha := \beta + 2 \, \varepsilon \, \alpha_{\theta},
\]
where $\alpha_{\theta} = \sin \theta(z) \, \dx + \cos \theta(z) \, \dy$. Then
\[
\drm \alpha = \underbrace{(1+2 \, \varepsilon \, \theta^\prime)}_{>0} \drm \beta 
\]
is again non-degenerate on $\eta$, and $\alpha(X) = \beta(X) + 2 \, \varepsilon \geq \varepsilon > 0$. Therefore, as $X$ is transverse to $\eta$ and $i_X \drm \alpha = 0$, it follows that $\alpha$ is a contact form with Reeb vector field $R_{\alpha} = (1/\alpha(X)) \, X$.

(ii) Let $M = T^2 \times \R$ with coordinates $(x,y,z)$ and choose a diffeomorphism $\varphi \col \R \xrightarrow{\cong} (-\pi/4, \pi/4)$. Define geodesic vector fields $X$ and $Y$ on $M$ by
\[
X = \partial_y + \partial_z, \quad Y = \sin \varphi(z) \, \partial_x + \cos \varphi(z) \, \partial_y.
\]
Then $Y$ induces a contact structure and $\langle X, Y \rangle = \cos \varphi(z) > 0$. But $X$ is transverse to the 2-torus $\{z = 0\}$, hence $X$ cannot be conformally Reeb.
\end{ex}
\input{bibliography}
\end{document}

%% file: packages.tex
\usepackage[latin1]{inputenc}
\usepackage{amssymb}
\usepackage{amsmath}
\usepackage{amsthm,epsfig}
\usepackage{hyperref}
\usepackage{color}
\usepackage{tikz-cd}
\usepackage{multirow}
\usepackage{graphicx}
\usepackage{accents}
\usepackage{pinlabel}
\usepackage{subcaption}
\usepackage[shortlabels]{enumitem}

%% file: commands.tex
\newtheorem{thm}{Theorem}
\newtheorem{lemma}[thm]{Lemma}
\newtheorem{cor}[thm]{Corollary}
\newtheorem{prop}[thm]{Proposition}
\theoremstyle{definition}

\newtheorem*{rmk}{Remark}
\newtheorem{ex}[thm]{Example}

\def\col{\colon\thinspace}

\newcommand{\R}{{\mathbb{R}}}

\newcommand{\Z}{{\mathbb{Z}}}
\newcommand{\N}{{\mathbb{N}}}
\newcommand{\E}{{\mathbb{E}}}

\newcommand{\vol}{\text{\rm vol}\,}

\newcommand{\abs}[1]{|#1|}
\newcommand{\set}[1]{\{#1\}}

\newcommand{\fol}{\mathcal{F}}
\newcommand{\lfol}{\tilde{\mathcal{F}}}

\newcommand{\mc}[1]{\mathcal{#1}}

\newcommand{\dx}{\mathrm{d}x}
\newcommand{\dy}{\mathrm{d}y}
\newcommand{\dz}{\mathrm{d}z}
\newcommand{\drm}{\mathrm{d}}
\newcommand{\xlif}{X_T}
\newcommand{\xdlif}{\accentset{\sim}{X}}
\newcommand{\ylif}{Y_T}

\newcommand{\alphtil}{\tilde{\alpha}}

\DeclareMathOperator{\intr}{\mathrm{Int}}

\DeclareMathOperator{\isom}{Isom}
\DeclareMathOperator{\rank}{rank}
\DeclareMathOperator{\sgn}{sgn}